\documentclass{elsarticle}

\usepackage{color}

\usepackage{amssymb,latexsym,amsmath, amsfonts}
\usepackage{amsthm,indentfirst}

\newtheorem{thm}{Theorem}
\newtheorem{lem}[thm]{Lemma}
\newtheorem{rem}[thm]{Remark}
\newtheorem{quest}{Question}
\newtheorem{prop}[thm]{Proposition}
\newtheorem{cor}[thm]{Corollary}
\newtheorem{dfn}[thm]{Definition}
\newtheorem{ex}[thm]{Example}

\begin{document}

\title{Measurable operators and the asymptotics of heat kernels and zeta functions}

\author{Alan Carey}
\ead{alan.carey@anu.edu.au}
\author{Fedor Sukochev}
\ead{f.sukochev@unsw.edu.au}

\cortext[cor1]{Corresponding Author}

\address{Mathematical Sciences Institute, Australian National University, Canberra ACT, 0200, Australia}

\address{School of Mathematics and Statistics, University of New South Wales, Sydney, 2052, Australia.}

\begin{abstract}
In this note we answer some questions inspired by the introduction in \cite{C1,C}, by Alain Connes, of the notion of measurable operators using Dixmier traces. These questions concern the relationship of measurability to the asymptotics
of $\zeta-$functions and heat kernels. The answers have remained elusive for some 15 years
\footnote{We dedicate this paper to the memory of Nigel Kalton.}.
\end{abstract}


\maketitle

\section{Introduction and Preliminaries}

In \cite{C} {(see also \cite{C1})} Alain Connes described in
part the relationship between Dixmier traces, heat kernel
asymptotics and the behaviour of $\zeta-$functions at their leading
singularity. In that discussion he introduced the notion of a
measurable operator. Subsequently these notions have arisen in
other contexts and interest has been generated in obtaining a
comprehensive picture of how they are related. The present authors
were forced to confront these ideas in their attempts to develop tools for
semifinite noncommutative geometry in \cite{CPS}, {\cite{CS}}.
Similar issues arise also in \cite{BF}.
In addition, after discussions with many colleagues, it
became clear to us that, for applications, extensions of \cite{C,
CPS} were needed. There has been considerable progress in the last
few years in \cite{BF, CRSS, KSS, LS, LSS, sed1, sed2, SS, SUZ, SZ}. In this
note we provide the final answer to two of the outstanding
questions.

As we have done previously in \cite{CPS}, {\cite{CS}} we will work  in the generality of
semifinite von Neumann algebras although even for the more
standard case of the bounded operators on Hilbert space the
results of this paper are  new. Let $\mathcal{M}$ be a von Neumann
algebra equipped with a faithful normal semifinite trace $\tau.$
For every operator $A\in\mathcal{M},$ let $E_{|A|}(s,\infty)$ denote
the spectral measure of
$|A|$, then its distribution function
$d_A$ and rearrangement $\mu(A)$ are defined by the following
formulas:
$$d_A(s)=\tau(E_{|A|}(s,\infty)),\ s>0$$
$$\mu(t,A)=\inf\{s:\ d_A(s)\leq t\},\ t>0.$$

The following sets of operators from $\mathcal{M}$ are widely used in noncommutative geometry (see \cite{C, BF, CPS, CRSS, CS, LS, LSS, SZ}). The reader should
be aware of the fact that the notation we are using is not that of \cite{C}.
$$\mathcal{M}_{1,\infty}=\{A\in\mathcal{M}:\ \mathop{\rm sup}_{t\in (0,\infty)}\frac1{\log(1+t)}\int_0^t\mu(s,A)ds<\infty\}$$
and
$$\mathcal{L}_{1,\infty}=\{A\in\mathcal{M}:\ \sup_{t>0}t\mu(t,A)<\infty\}.$$
Equipped with the norm
$$
\|A\|_{\mathcal{M}_{1,\infty}}:=\mathop{\rm sup}_{t\in (0,\infty)}\frac{1}{\log(1+t)}\int _0^t\mu(s,A)ds
$$
the first set is an example of a Marcinkiewicz operator space. The
second set is the so-called weak $\mathcal{L}_1$ space, which is a
linear (non-closed) subspace in $(\mathcal{M}_{1,\infty},
\|\cdot\|_{\mathcal{M}_{1,\infty}})$. Recall also that
$\mathcal{L}_{1,\infty}$ is not dense in $\mathcal{M}_{1,\infty}$
with respect to the norm $\|\cdot\|_{\mathcal{M}_{1,\infty}}$
{(see e.g. \cite[Lemma 5.5 in Ch.II.7]{KPS})}.

We need the (multiplicative) Cesaro operator acting on the space $L_{\infty}(0,\infty)$ of all essentially bounded Lebesgue measurable functions given  by the formula
\begin{equation} \label{Cesaro}(Mx)(\nu)=\frac1{\log(\nu)}\int_1^\nu x(s)\frac{ds}{s}.
\end{equation}
If $A\in\mathcal{L}_{1,\infty},$ then it follows from \cite[Lemma 5.1]{CRSS} that
\begin{equation}\label{hk weakl1}
\sup_{\lambda}\frac1{\lambda}\tau(e^{-(\lambda A)^{-1}})<\infty.
\end{equation}
The inequality \eqref{hk weakl1} does not necessarily hold for $A\in\mathcal{M}_{1,\infty}$ (see  \cite[Example on p. 274]{CRSS}). It is implicitly proved in \cite{CRSS} (see also \cite[Theorem 40 and Corollary 41]{SZ} where a much stronger result is established) that
\begin{equation}\label{hk marc}
\sup_{\lambda}M(\lambda\to\frac1{\lambda}\tau(e^{-(\lambda A)^{-1}}))<\infty
\end{equation}
for every $A\in\mathcal{M}_{1,\infty}$
where the notation is a shorthand for taking the supremum of the function obtained from 
applying $M$ to $\lambda\to\frac1{\lambda}\tau(e^{-(\lambda A)^{-1}})$. 

This note is motivated by the following two questions (that we will completely answer here).

\begin{quest}\label{CKS quest1} Suppose that $A\in\mathcal{L}^+_{1,\infty}$ is such that the  limit
$$\lim_{\lambda\to\infty}\frac1{\lambda}\tau(e^{-(\lambda A)^{-1}})$$
exists. What information is then available on the distribution function of the operator $A$?
\end{quest}

\begin{quest}\label{CKS quest2} Suppose that $A\in\mathcal{M}^+_{1,\infty}$ is such that the following limit
$$\lim_{\lambda\to\infty}M(\lambda\to\frac1{\lambda}\tau(e^{-(\lambda A)^{-1}}))$$
exists. What information is then available on the distribution function of the operator $A$?
\end{quest}
Note we are again using an obvious  shorthand notation in Question 2.
These questions are in fact related to somewhat similar matters
studied in \cite{CRSS} (there they are called questions A and B)
for the $\zeta$-function and the connection between them can be
established via the theory of Dixmier traces (see e.g. \cite{D,C,
CPS, BF, CS, LSS, LS, KSS, SS, SUZ}). Very briefly, we now recall
some basic definitions from that theory.

First, a positive normalised functional on a unital von Neumann algebra is called a state and any state
on the algebra $ L_{\infty}(0,\infty)$ is called a generalised limit if it vanishes on every function
with compact support.

Second, given $s>0$, a  dilation operator $\sigma_s:L_{\infty}(0,\infty)\to L_{\infty}(0,\infty)$
is defined by setting $(\sigma_sx)(t)=x(t/s)$.  A generalised limit $\omega$ is said to be dilation
invariant if $\omega\circ\sigma_s=\omega$ for every $s>0$.

Third, if $\omega$ is an arbitrary dilation invariant generalised limit then a Dixmier trace
$\tau_{\omega}$ on $\mathcal{M}_{1,\infty}$ is defined (see \cite[Definition 9]{KSS}) by the formula
$$\tau_{\omega}(A):=\omega(t\to\frac1{\log(1+t)}\int_0^t\mu(s,A)ds),\quad A\in\mathcal{M}^+_{1,\infty}.$$
Recall that a positive linear functional $\varphi$ on $\mathcal M_{1,\infty}$ is called fully symmetric if
for all $0\leq A,B\in\mathcal{M}_{1,\infty}$ such that
\[
\int_{0}^t\mu(s,B)ds\leq\int_{0}^t\mu(s,A)ds,\quad \mbox{for all } t>0,
\]
we have $\varphi(B)\leq\varphi(A)$. { In this note the largest possible class of Dixmier traces,
namely the class
$$
\mathcal{D}:=\{\tau_\omega: \ \omega\ {\rm is\ a \ dilation\ invariant\ generalized\ limit}\}
$$
of all Dixmier traces is needed. (See further possibilities in \cite{C, LSS,
CS, LS, SUZ, SS}). That this class is natural is confirmed by the
fact that $\mathcal{D}$ coincides with the class of all fully
symmetric singular functionals on $\mathcal M_{1,\infty}$. }

More precisely, the following assertion follows from
\cite[Theorem 11]{KSS}  if we set the function 
denoted by $\psi$ in that theorem to be $\psi(t)=\log(1+t)$. 

\begin{thm}\label{kss all theorem} For every fully symmetric functional $\varphi$ on $\mathcal{M}_{1,\infty},$
there exists a dilation invariant generalised limit $\omega$ such that the Dixmier trace $\tau_{\omega}=\varphi.$
\end{thm}

Now we establish the notation for, and background to, our main theorem.

Let $\omega$ be an arbitrary dilation invariant generalised limit. A heat kernel functional $\xi_{\omega}$ is defined (see \cite[Sections 1 and 5]{SZ}) by the formula
$$\xi_{\omega}(A):=(\omega\circ M)(\lambda\to \frac1{\lambda}\tau(e^{-(\lambda A)^{-1}})),\quad A\in\mathcal{M}^+_{1,\infty}.$$
See  \cite{C, CPS, CRSS, SZ}) for the reasons for this particular form of the definition of the heat kernel functional.

Let $\gamma$ be an arbitrary generalised limit. The $\zeta$-function residue (associated with $\gamma$) is defined (see \cite[Section 1]{SZ} and also \cite{C, CPS}) by the formula
$$\zeta_{\gamma}(A):=\gamma(r\to\frac1r\tau(A^{1+1/r})).$$

Evidence that the zeta and heat kernel functionals are
closely related comes from the  following theorems, proved in \cite{SZ}.

\begin{thm}\label{zeta lin fs}\cite[Theorem 8]{SZ}. For every generalised limit $\gamma$, the $\zeta$-function residue $\zeta_{\gamma}$ is a fully symmetric functional on $\mathcal{M}_{1,\infty}$.
\end{thm}

\begin{thm} \cite[Theorem 22]{SZ}. For every dilation invariant generalised limit $\omega$, the heat kernel functional $\xi_{\omega}$ is a fully symmetric functional on $\mathcal{M}_{1,\infty}$.
\end{thm}

\begin{thm}\label{sz all theorem} \cite[Theorem 31]{SZ}.
For every fully symmetric functional $\varphi$ on $\mathcal{M}_{1,\infty}$,
there exists a dilation invariant generalised limit $\omega$ such that the heat kernel functional $\xi_{\omega}=\varphi$.
\end{thm}

\begin{rem}\label{constant} In fact, it is proved in \cite[Theorem 31]{SZ} and \cite[Lemma 20]{SZ} that for every fully  symmetric functional $\varphi$ on $\mathcal{M}_{1,\infty}$,
there exists a dilation invariant generalised limit $\omega$ such that for every $q>0$, $$(\omega\circ M)(\lambda\to\frac1{\lambda}\tau(e^{-(\lambda A)^{-q}}))=\Gamma (1+1/q)\varphi$$.
\end{rem}
In view of Theorems \ref{kss all theorem} to \ref{sz all theorem}, it is natural to ask whether the equality $\tau_{\omega}=\xi_{\omega}$ holds for an arbitrary dilation invariant generalised limit $\omega$. This is however not the case, see \cite[Theorem 37]{SZ} where examples of $\omega$'s are  given for which we have $\tau_{\omega}\neq \xi_{\omega}$.

Finally, we come to one of the major new notions introduced in this context in \cite{C} (see also
\cite{C1}) and generalised in \cite{LSS, CS, LS, SS, SUZ}:
\begin{dfn}\label{meas}
 The operator $A\in\mathcal{M}_{1,\infty}$ is said to be measurable if and only if the set
 $\{\tau_{\omega}(A):\ \tau_\omega\in\mathcal{D}\}$ consists of a single point.
\end{dfn}

In view of the previously cited results and counter-examples our main result, which we now state,
is not entirely expected. It answers Question \ref{CKS quest2} and complements and extends earlier results in \cite{C, CPS, CRSS}. It also provides a very short new proof of the main result from \cite{LSS} (see the proof of the implication ${\rm (i)}\Longrightarrow {\rm (ii)}$ below).
\begin{thm}\label{main eq teor} Let $A\in\mathcal{M}_{1,\infty}$ be a positive operator. The following conditions are equivalent.
\begin{enumerate}
\item[{\rm (i)}]\label{sx1} The operator $A$ is measurable.
\item[{\rm (ii)}]\label{sx2} The limit
$\lim_{t\to\infty}\frac1{\log(1+t)}\int_0^t\mu(s,A)ds$ exists.
\item[{\rm (iii)}]\label{sx3} The limit
$\lim_{\lambda\to\infty}M(\lambda\to \frac1{\lambda}\tau(e^{-(\lambda A)^{-1}}))$ exists.
\item[{\rm (iv)}]\label{sx4} The limit
$\lim_{s\to 0}s\tau(A^{1+s})$ exists.
\end{enumerate}
Furthermore, if any of the conditions (i)-(iv) above holds, then we have the coincidence of the three limits
$$\lim_{t\to\infty}\frac1{\log(1+t)}\int_0^t\mu(s,A)ds=\lim_{\lambda\to\infty}M(\lambda\to\frac1{\lambda}\tau(e^{-(\lambda A)^{-1}})) =\lim_{s\to 0}s\tau(A^{1+s})$$
with the value given by $\{\tau_{\omega}(A):\ \tau_\omega\in\mathcal{D}\}.$
\end{thm}

\begin{rem} Let $a_1=\tau_{\omega}(A)$ for every $\tau_\omega\in\mathcal{D}$ in $(i)$ in Theorem \ref{main eq teor}. Let $a_2,$ $a_3$ and $a_4$ be the limits in $(ii),$ $(iii)$ and $(iv)$ (respectively) in Theorem \ref{main eq teor}. For every $1\leq i,j\leq 4,$ we show in the proof of Theorem \ref{main eq teor} that $(i)\implies(j)$ and $a_i=a_j.$
\end{rem}

The results should be seen in the general context of the continuing study of the notion
 of measurable operators introduced in \cite{C1,C} and further
 elaborated in \cite{LSS, CS, LS}. The main interest
 remains in the following areas: (i) comparing various modifications
 of this notion with respect to various subsets of Dixmier traces
 (as a rule with additional properties of invariance), (ii) finding
 convenient descriptions of the set of self-adjoint measurable
 operators, and (iii) determining when a given self-adjoint measurable
 operator is Tauberian. We remark that this current note is related to
 progress on these directions which will appear in \cite{SS, SUZ},
 where it is shown that not every self-adjoint measurable operator is necessarily Tauberian (which
  is in stark contrast with the case of positive operators). It will also be shown in \cite{SUZ} that the
  notion of measurability as originally introduced by Connes in \cite{C1}
  and its version considered in \cite{LSS} actually coincide.

\noindent{\bf Acknowledgements}. This research was supported by the Australian
Research Council.
The authors thank Dima Zanin for numerous discussions and help in the preparation of this article. The idea to use Lemma \ref{second tauberian lemma} below belongs to him and the usage of this lemma has significantly simplified our original proofs. We also thank  Bruno Iochum for many discussions on the issues
surrounding the results of this note. The first named author thanks the Alexander von Humboldt Stiftung and colleagues at the University of M\"unster.

\section{Proof of the main result}

The following lemma is well-known. The proof can be found in e.g. \cite[Section 6.8]{Hardy}.
\begin{lem}\label{first tauberian lemma} Let $z\in L_{\infty}(0,\infty)$ be a positive differentiable function. If $tz'(t)\geq{\rm const}$ for every $t>0$, then the following implication holds
$$\lim_{t\to\infty}\frac1t\int_0^tz(s)ds=C\Longrightarrow\lim_{t\to\infty}z(t)=C.$$
\end{lem}

The following lemma is also well-known. Due to the lack of a suitable reference  we provide a short proof for convenience of the reader.

\begin{lem}\label{gen lim lemma} Let $x\in L_{\infty}(0,\infty)$ and let $a\in\mathbb{R}.$ The following conditions are equivalent.
\begin{enumerate}
\item\label{gen lim lemma1}  We have the bounds$$\liminf_{t\to\infty}x(t)\leq a\leq\limsup_{t\to\infty}x(t).$$
\item\label{gen lim lemma2} There exists a generalised limit $\gamma$ such that $\gamma(x)=a.$
\end{enumerate}
\end{lem}
\begin{proof} The implication $\eqref{gen lim lemma2}\to\eqref{gen lim lemma1}$ follows immediately from the definition of the generalised limit.

In order to prove the implication $\eqref{gen lim lemma1}\to\eqref{gen lim lemma2},$ define a functional $\gamma$ on $\mathbb{R}+x\mathbb{R}$ by setting $\gamma(\alpha+\beta x)=\alpha+\beta a.$ Clearly,
$$\gamma(z)\leq\limsup_{t\to\infty}z(t),\quad z\in\mathbb{R}+x\mathbb{R}.$$
The assertion follows now from the Hahn-Banach theorem.
\end{proof}

Our next lemma plays an important role in the proof of our main result.

\begin{lem}\label{second tauberian lemma} Let $z$ be a positive locally integrable function on $(0,\infty)$. If $Mz\in L_{\infty}(0,\infty),$ then we have
$$\lim_{t\to\infty}(M^2z)(t)=C\Longrightarrow\lim_{t\to\infty}(Mz)(t)=C.$$
\end{lem}
\begin{proof} Set $x=(Mz)\circ\exp$. We have
$$(M^2z)(t)=\frac1{\log(t)}\int_1^t(Mz)(u)\frac{du}{u}{=}\frac1{\log(t)}\int_0^{\log(t)}x(s)ds,$$
where we used the substitution ${u=e^s}$ in the second equality.
By the assumption, we have
$$\lim_{t\to\infty}\frac1t\int_0^tx(s)ds=C.$$
Let us now verify that the function $t\to tx'(t)$ satisfies the assumption of Lemma \ref{first tauberian lemma}. We have
$$tx'(t)=t(\frac1t\int_0^{e^t}z(s)\frac{ds}{s})'=-\frac1t\int_0^{e^t}z(s)\frac{ds}{s}+z(e^t).$$
Since $z$ is positive, we have $tx'(t)\geq -(Mz)(e^t)$ and since $Mz\in L_{\infty}(0,\infty)$, we conclude $tx'(t)\geq{\rm const}$. By Lemma \ref{first tauberian lemma}, we have $\lim_{t\to\infty}x(t)=C$ and hence $\lim_{t\to\infty}(Mz)(t)=C$.
\end{proof}

The following remark is well known and can be found in e.g. \cite{C}.

\begin{rem}\label{cdix}
For every generalised limit $\gamma,$ the state $\gamma\circ M$ is a dilation invariant generalised limit.
\end{rem}

With these preliminary results in hand we come to the proof of our main result.
\begin{proof} (Of  Theorem \ref{main eq teor}.) First, the implication ${\rm (ii)}\Longrightarrow{\rm (i)}$ follows from the definition of $\tau_{\omega}.$
Next, the implication ${\rm (i)}\Longrightarrow {\rm (ii)}$ was first proved  in \cite[Theorem 6.6]{LSS} (see also \cite{CS}). We provide here a new (very short and straightforward) proof.

Let
$$C:=\tau_{\omega}(A),\ \mbox{for all } \tau_\omega\in\mathcal{D}.$$
In particular, by Remark \ref{cdix}, we have $\tau_{\gamma\circ M}(A)=C$ for every generalised limit $\gamma$. That is, we have the equality
$$(\gamma\circ M)(t\to\frac1{\log(1+t)}\int_0^t\mu(s,A)ds)=C,$$
which, due to Lemma \ref{gen lim lemma}, guarantees
\begin{equation}\label{gfrew1}
\lim_{t\to\infty}M(t\to\frac1{\log(1+t)}\int_0^t\mu(s,A)ds)=C.
\end{equation}
Set $z(t):=t\mu(t,A)$. Observe that $z$ is a positive measurable, but not necessarily bounded function. However, since $A\in\mathcal{M}_{1,\infty}$, the function
$$t\to (Mz)(t)=\frac1{\log(t)}\int_1^t\mu(s,A)ds$$
is bounded. Thus, $Mz\in L_{\infty}(0,\infty)$ and obviously
\begin{equation}\label{gfrew2}
\lim_{t\to\infty}(Mz)(t)-\frac1{\log(1+t)}\int_0^t\mu(s,A)=0.
\end{equation}
Combining \eqref{gfrew1} and \eqref{gfrew2}, and using the (obvious) fact that $\lim_{t\to\infty}(My)(t)=0$ whenever $y\in L_\infty(0,\infty)$ satisfies $\lim_{t\to\infty}y(t)=0$,  we infer that
$$\lim_{t\to\infty}(M^2z)(t)=C.$$
By Lemma \ref{second tauberian lemma}, we obtain from the preceding equality
$$\lim_{t\to\infty}(Mz)(t)=C$$
and the proof of the implication is completed by referring to \eqref{gfrew2}.

${\rm (iii)}\Longrightarrow {\rm (i)}$. Let $C$ be the limit in {\rm (iii)}. By definition of $\xi_{\omega}$, we have $\xi_{\omega}(A)=C$ for every dilation invariant generalised limit $\omega$.
By Theorems \ref{kss all theorem} and \ref{sz all theorem}, the class $\mathcal{D}$ coincides with the class of all heat kernel functionals and so, we also have $\tau_{\omega}(A)=C$ for every $\tau_{\omega}\in  \mathcal{D}$ and the proof of the implication is completed.

${\rm (i)}\Longrightarrow {\rm (iii)}$. Suppose that $\tau_{\omega}(A)=C$ for every $\tau_{\omega}\in  \mathcal{D}$. Then the same argument as above shows that $\xi_{\omega}(A)=C$ for every dilation invariant generalised limit $\omega$. In particular, due to Remark \ref{cdix}, we have $\xi_{\gamma\circ M}(A)=C$ for every generalised limit $\gamma$. That is,
$$(\gamma\circ M^2)(\lambda\to\frac1{\lambda}\tau(e^{-(\lambda A)^{-1}}))=C.$$
It follows from Lemma \ref{gen lim lemma} that
$$\lim_{\lambda\to\infty}M^2(\lambda\to \frac1{\lambda}\tau(e^{-(\lambda A)^{-1}}))=C.$$
Due to \eqref{hk marc}, we know that the mapping $\lambda\to M(\lambda\to\frac1{\lambda}\tau(e^{-(\lambda A)^{-1}}))$ is bounded and therefore
the proof of the implication is completed by invoking Lemma \ref{second tauberian lemma}.

${\rm (i)}\Longrightarrow {\rm (iv)}$. Suppose that $\tau_{\omega}(A)=C$ for every $\tau_{\omega}\in  \mathcal{D}$. It follows from Theorems \ref{zeta lin fs} and \ref{kss all theorem} that the class of all $\zeta-$function residues is a subclass of $\mathcal{D}$. Hence, for every generalised limit $\gamma,$ we have
$$\gamma(t\to\frac1t\tau(A^{1+1/t}))=C.$$
An appeal to Lemma \ref{second tauberian lemma} completes the proof of the implication.

Finally, the implication ${\rm (iv)}\Longrightarrow{\rm (i)}$ is established in \cite[Theorem 3.1]{CPS}.
\end{proof}

Our methods have a further interesting consequence.
 Repeating the argument $(i)\Longrightarrow(iii)$ in Theorem \ref{main eq teor} verbatim (and using Remark \ref{constant} instead of Theorem \ref{sz all theorem}), we obtain the following result.
\begin{prop}\label{obyasn prop} For every measurable positive operator $A\in\mathcal{M}_{1,\infty}$ and every $q>0,$ we have
$$\lim_{\lambda\to\infty}M(\lambda\to\frac1{\lambda}\tau(e^{-(\lambda A)^{-q}}))=\Gamma(1+\frac1q)\lim_{s\to0}s\tau(A^{1+s}).$$
\end{prop}

\section{Answering Question \ref{CKS quest1}}

The following corollary answers Question \ref{CKS quest1}. Its proof immediately follows from the implication ${\rm (iii)}\Longrightarrow{\rm (ii)}$ established in Theorem \ref{main eq teor}.

\begin{cor} \label{q2} Let $A\in\mathcal{L}_{1,\infty}$ be a positive operator. If  the limit
$$\lim_{\lambda\to\infty}\frac1{\lambda}\tau(e^{-(\lambda A)^{-1}})$$
exists then so does the limit
$$\lim_{t\to\infty}\frac1{\log(1+t)}\int_0^t\mu(s,A)ds.$$
\end{cor}
The following example shows that the converse to Corollary \ref{q2} does not hold. 

\begin{ex}\label{ex} There exists a positive operator $A\in\mathcal{L}_{1,\infty}$ such that
\begin{equation}\label{0}\lim_{t\to\infty}\frac1{\log(1+t)}\int_0^t\mu(s,A)ds=0\end{equation}
and
\begin{equation}\label{e}
\limsup_{t\to\infty}\frac1{\lambda}\tau(e^{-(\lambda A)^{-1}})>0.
\end{equation}
\end{ex}
\begin{proof} Define a positive operator $A$ by setting
$$\mu(s,A)=
\begin{cases}
s^{-1},\qquad s\in(e^{e^n},ne^{e^n}),\ n\geq1\\
e^{-e^{n+1}},\quad s\in(ne^{e^n},e^{e^{n+1}}),\ n\geq1\\
e^{-e},\qquad s\in(0,e^e).
\end{cases}
$$

For every $n\geq1,$ we have
$$\int_0^{e^{e^{n+1}}}\mu(s,A)ds=1+\sum_{k=1}^n\left(\int_{e^{e^k}}^{ke^{e^k}}\mu(s,A)ds+\int_{ke^{e^k}}^{e^{e^{k+1}}}\mu(s,A)ds\right)=$$
$$=1+\sum_{k=1}^n\left(\log(k)+1-(k+1)e^{-(e-1)e^k}\right)=n+\log(n!)+O(1)=O(n\log(n)).$$
Here, the last equality follows from Stirling's formula
$$
n!=\sqrt{2\pi n}\left(\frac{n}e\right)^ne^{\frac{\theta}{12n}},\quad 0<\theta<1.$$

For every $t>e,$ let $\nu=\nu(t)=[\log(\log(t))].$ It follows that
$$\int_0^t\mu(s,A)ds\leq\int_0^{e^{e^{\nu+1}}}\mu(s,A)ds=O(\nu\log(\nu))=o(\log(t)),$$
which yields \eqref{0}.

On the other hand, we have
$$\frac1{\lambda}\tau(e^{-(\lambda A)^{-1}})\geq\frac1{\lambda}\sum_{n=1}^{\infty}\int_{e^{e^n}}^{ne^{e^n}}e^{-\lambda^{-1}s}ds=\sum_{n=1}^{\infty}e^{-\lambda^{-1}e^{e^n}}-e^{-n\lambda^{-1}e^{e^n}}.$$
For a given $n\in\mathbb{N},$ set $\lambda=e^{e^n}.$ It follows that
$$\frac1{\lambda}\tau(e^{-(\lambda A)^{-1}})\geq e^{-\lambda^{-1}e^{e^n}}-e^{-n\lambda^{-1}e^{e^n}}=e^{-1}-e^{-n}.$$
Therefore,
$$\limsup_{\lambda\to\infty}\frac1{\lambda}\tau(e^{-(\lambda A)^{-1}})\geq e^{-1},$$
yielding \eqref{e}.
\end{proof}

This example has a further interesting consequence.
\begin{cor}
The limit
\begin{equation}\label{e1}
\lim_{t\to\infty}\frac1{\lambda}\tau(e^{-(\lambda A)^{-1}})
\end{equation}
does not exist and hence we cannot omit $M$ in Theorem \ref{main eq teor}.
\end{cor}
\begin{proof}
Suppose that the limit in \eqref{e1} exists and is equal to $c$. Then, obviously
$$\lim_{\lambda\to\infty}M(\lambda\to\frac1{\lambda}\tau(e^{-(\lambda A)^{-1}}))=c$$
and by Theorem \ref{main eq teor}, we obtain that $$\lim_{t\to\infty}\frac1{\log(1+t)}\int_0^t\mu(s,A)ds =c.$$
It follows from \eqref{0} that $c=0$. Thus, we should then have that the limit in \eqref{e1} is $0$.
However, the latter contradicts \eqref{e}.
\end{proof}

Finally we see that this example demonstrates that we are not able to claim
any meromorphic continuation property for the zeta function on the basis of our
results to this point.

\begin{lem}\label{mero} For the operator $A$ constructed in Example \ref{ex}, the $\zeta$-function $s\to \tau(A^{1+s})$ does not have a pole or a removable singularity at $0$.
\end{lem}
\begin{proof} Assume the contrary, that is, the $\zeta$-function admits an analytic continuation into the punctured neighborhood of $0$ and has an $n$-th order pole there. By Theorem \ref{main eq teor} and \eqref{0}, we have $\lim_{s\to 0}s\tau(A^{1+s})=0$. Therefore, we have $\lim_{s\to 0}s^n\tau(A^{1+s})=0$, which contradicts the assumption. The $\zeta$-function $s\to\tau(A^{1+s})$ does not have a removable singularity  at $0$ because $A\notin\mathcal{L}_1$ (that is the limit $\lim_{s\to 0}\tau(A^{1+s})$ does not exist).
\end{proof}

It is important to observe that we are also in a position to answer analogues of
Questions \ref{CKS quest1} and \ref{CKS quest2} in the case of arbitrary operators from
$\mathcal{M}_{1,\infty}$ (not necessarily positive). For brevity, we state and prove such analogues
for self-adjoint operators.

\begin{thm}\label{extra1} Suppose that a self-adjoint operator $A\in\mathcal{M}_{1,\infty}$
is such that the following limit
$$\lim_{\lambda\to\infty}M(\lambda\to\frac1{\lambda}(\tau(e^{-(\lambda A_+)^{-1}})-\tau(e^{-(\lambda A_+)^{-1}})))$$
exists. Then the operator $A$ is measurable.
\end{thm}
\begin{proof} The proof is a verbatim repetition of the arguments used in the proof of the  implication
${\rm (iii)}\Longrightarrow{\rm (i)}$ in Theorem \ref{main eq teor}. We omit further details.
\end{proof}

It is worth remarking  that we cannot ascertain whether, under the assumptions in
Theorem \ref{extra1}, the limit
$$\lim_{t\to\infty}\frac1{\log(1+t)}\int_0^t(\mu(s,A_+)-\mu(s,A_-))ds$$
exists. However, this can be done, if $A$ belongs to the weak $\mathcal{L}_{1}$ space.

\begin{thm}\label{extra2} Suppose that a self-adjoint operator $A\in\mathcal{L}_{1,\infty}$ is such that the following limit
$$\lim_{\lambda\to\infty}\frac1{\lambda}(\tau(e^{-(\lambda A_+)^{-1}})-\tau(e^{-(\lambda A_+)^{-1}}))$$
exists. Then the operator $A$ is measurable, and, in addition, the limit
$$\lim_{t\to\infty}\frac1{\log(1+t)}\int_0^t(\mu(s,A_+)-\mu(s,A_-))ds$$
exists.
\end{thm}
\begin{proof} The first assertion follows immediately from Theorem \ref{extra1}.
The second assertion is provided by { \cite[Corollary 19]{SS}.}
\end{proof}
\begin{rem}
\label{more}The assumption $A\in\mathcal{L}_{1,\infty}$ in Theorem \ref{extra2} above can be further weakened by requesting $\mu(t;A)=o(\frac{log(1+t)}{t})$ for sufficiently large $t>0$. In a sense the latter is the best possible, in particular, the assertion of Theorem \ref{extra2} fails if the latter condition does not hold. For details, we refer the reader to \cite{SS}.
\end{rem}

\section{The case $p>1$ and examples}

\subsection{Notations}

We firstly say a few words concerning the notations.

In the paper \cite{C1}(where the applications of Dixmier traces
to noncommutative geometry were first presented) Alain Connes
considered the ideal ${\mathcal L}^{1+}$ of all compact operators
$T$ on an infinite-dimensional Hilbert space ${H}$ whose singular
values $\{\mu(j, T)\}_{j\in \mathbb N}$ satisfy
$$
\sup_{N>1}\frac{1}{\log N}\sum _{j=1}^N\mu(j, T)<\infty.
$$
This ideal later, in \cite[p.303]{C} was denoted by ${\mathcal
L}^{(1,\infty)}$. Further,  in \cite[p.677]{C1},
 the ideal ${\mathcal L}^{n+}$, whose $n^{th}$ root lies in ${\mathcal
L}^{1+}$ in $B(H)$, was introduced.
It is noted in \cite{C} that the ideals ${\mathcal L}^{(p,\infty)}$
correspond to the notion of weak ${\mathcal L}^p$-spaces in
classical analysis. An alternative notation ${\mathcal
L}^{p+}$ is also mentioned.

It is now important to realize that there is a small notational
discrepancy here, and addressing this discrepancy, we have used
another notation for the space ${\mathcal L}^{1+}$ in \cite{C1}
and ${\mathcal L}^{(1,\infty)}$ in \cite{C}. Namely, we used the
symbol $\mathcal{M}_{1,\infty}$. We now explain a little bit more
about our choice.

As noted in \cite{C}, the Banach space
$(\mathcal{M}_{1,\infty},\|\cdot\|_{\mathcal{M}_{1,\infty}})$ was
probably first considered by Macaev \cite{Mat} (with yet another
notation, which we do not use here at all  in order not to confuse the
reader) as the dual space to the ideal which is customarily
called, a {\it Macaev ideal}. For a complete exposition of
the theory of these spaces and detailed references, we refer the
reader to the books \cite{GK1, GK2}.  The reason we used
this notation is due to the fact that the space
$(\mathcal{M}_{1,\infty},\|\cdot\|_{\mathcal{M}_{1,\infty}})$ may
be viewed as a noncommutative analogue of a Sargent (sequence)
space, see \cite{Sar}. This fact is explained in the
article \cite{Pie} by A.Pietsch, for which we refer the reader for
a fuller treatment of the history of the space
$\mathcal{M}_{1,\infty}$ and additional references. We follow
this notation also because it allows us to
reserve $\mathcal{L}_{1,\infty}$ for the
well-established notion of quasi-normed weak $L_1$-space (which
we identify here with a non-closed subspace in
$(\mathcal{M}_{1,\infty},\|\cdot\|_{\mathcal{M}_{1,\infty}})$).

The classical  $p$-convexification procedure for an arbitrary
Banach lattice $X$ is described in \cite[Section 1.d]{LT2} and is
sometimes termed power norm transformation. It is simply a direct
generalization of the procedure of defining $L_p$-spaces from an
$L_1$-space. Applying the analogous operation to the ideal
$\mathcal{M}_{1,\infty}$, we obtain the space $\mathcal{Z}_p$
firstly introduced and (alternatively) described in \cite{CRSS}.
It is unfortunate that, due to other notations used in
\cite{CRSS}, the space $\mathcal{Z}_p$ was identified there with
the notation ${\mathcal L}_{p,\infty}$. One of the reasons, we
have switched to the notations $\mathcal{M}_{1,\infty}$ and
$\mathcal{L}_{1,\infty}$ is that the notation ${\mathcal
L}_{p,\infty}$ is then properly associated with the
$p$-convexification of the weak $L_1$ space
$\mathcal{L}_{1,\infty}$. In this way, our usage of the symbol
$\mathcal{L}_{p,\infty}$ is perfectly compatible with the usage of
the same symbol in \cite{C} for all $p>1$ (excepting $p=1$ for
which we use $\mathcal{M}_{1,\infty}$). It is now natural to
denote the space $\mathcal{Z}_p$ by the symbol
$\mathcal{M}_{p,\infty}$.  Thus, the space
$\mathcal{M}_{p,\infty}$ is exactly obtained by asking for $p$-th roots in
$\mathcal{M}_{1,\infty}$ and coincides with the space $L^{p+}$
from \cite{C1}, whereas the $p$-convexification of its subspace
$\mathcal{L}_{1,\infty}$ equipped with the weak quasi-norm yields
the Banach space $\mathcal{L}_{p,\infty}$ and this is exactly the
same space from \cite{C} which we cited above, at the beginning of
this subsection.

\subsection{Results}
The following assertion is a consequence of Theorem \ref{main eq
teor}.

\begin{cor}\label{main eq teorrrr} Let $A\in\mathcal{Z}_p=\mathcal{M}_{p,\infty}$ be a positive operator.
The following conditions are equivalent.
\begin{enumerate}
\item[{\rm (i)}]\label{sx11} The operator $A^p$ is measurable.
\item[{\rm (ii)}]\label{sx12} The limit
$\lim_{t\to\infty}\frac1{\log(1+t)}\int_0^t\mu(s,A^p)ds$ exists.
\item[{\rm (iii)}]\label{sx13} The limit
$\lim_{\lambda\to\infty}M(\lambda\to\frac1{\lambda^p}\tau(e^{-(\lambda A)^{-p}}))$ exists.
\item[{\rm (iv)}]\label{sx14} The limit
$\lim_{s\to 0}s\tau(A^{p+s})$ exists.
\end{enumerate}
Furthermore, if any of the conditions (i)-(iv) above holds, then we have the coincidence of the three limits
$$\lim_{t\to\infty}\frac1{\log(1+t)}\int_0^t\mu(s,A^p)ds=\lim_{\lambda\to\infty}M(\lambda\to\frac1{\lambda^p}\tau(e^{-(\lambda A)^{-p}})) =\frac1p\lim_{s\to 0}s\tau(A^{p+s})$$
with the value given by $\{\tau_{\omega}(A^p):\ \tau_\omega\in\mathcal{D}\}.$
\end{cor}
\begin{proof} {Set $B=A^p.$ Clearly, $B\in\mathcal{M}_{1,\infty}$ and
$$\lim_{s\to0}s\tau(A^{p+s})=p\lim_{s\to0}s\tau(B^{1+s}).$$}
Let $P:L_{\infty}(0,\infty)\to L_{\infty}(0,\infty)$ be the operator defined by setting $(Px)(t)=x(t^p),$ $x\in L_{\infty}(0,\infty),$ $t>0.$ We have $PM=MP$ (see \cite[Proposition 1.3(4)]{CPS}). Hence,
$$\lim_{\lambda\to\infty}M(\lambda\to\frac1{\lambda^p}\tau(e^{-(\lambda A)^{-p}}))=\lim_{\lambda\to\infty}MP(\lambda\to\frac1{\lambda}\tau(e^{-(\lambda B)^{-1}}))=$$
$$=\lim_{\lambda\to\infty}PM(\lambda\to\frac1{\lambda}\tau(e^{-(\lambda B)^{-1}}))=\lim_{\lambda\to\infty}M(\lambda\to\frac1{\lambda}\tau(e^{-(\lambda B)^{-1}})).$$
From this last equality it is clear that the result follows immediately from Theorem \ref{main eq teor}.
\end{proof}

\begin{rem} The implication $(iii)\to(i)$ of Theorem \ref{main eq teorrrr} significantly strengthens Proposition 5.3 of \cite{CRSS}.
This is because we require here only the existence of the limit in $(iii)$ and not an asymptotic expansion for the heat kernel as is assumed in \cite{CRSS} and furthermore we do not require $\omega$ to be $M-$invariant as is needed in \cite{CRSS}.
\end{rem}

The next corollary prepares the way for a discussion of heat kernel bounds. It follows from Proposition \ref{obyasn prop} and parallels \cite[Proposition 4.2]{CPS}. The latter proposition looks similar to the one below, however its proof is totally different. 

\begin{cor}\label{parallel} Let $A\in\mathcal{Z}_p=\mathcal{M}_{p,\infty}$ be a positive operator such that $A^p$ is measurable. We have
$$\lim_{\lambda\to\infty}M(\lambda\to\frac1{\lambda^p}\tau(e^{-(\lambda A)^{-2}}))=\frac12\Gamma(\frac{p}2)\lim_{s\to0}s\tau(A^{p+s}).$$
\end{cor}
\begin{proof} Set $B=A^p.$ Clearly, $B\in\mathcal{M}_{1,\infty}$. Using the same argument as in the proof of Corollary \ref{main eq teorrrr}, we obtain
$$\lim_{\lambda\to\infty}M(\lambda\to\frac1{\lambda^p}\tau(e^{-(\lambda A)^{-2}}))= \lim_{\lambda\to\infty}M(\lambda\to\frac1{\lambda}\tau(e^{-\lambda^{-2/p}A^{-2}})).$$
Now, by Proposition \ref{obyasn prop}, we have
$$\lim_{\lambda\to\infty}M(\lambda\to\frac1{\lambda}\tau(e^{-(\lambda B)^{-2/p}}))=\Gamma(1+\frac{p}2)\lim_{s\to0}s\tau(B^{1+s}).$$
Finally, we write
$$\lim_{s\to0}s\tau(B^{1+s})=\frac1p\lim_{s\to0}s\tau(A^{p+s}).$$
\end{proof}

%
%
\subsection{Discussion}
It would of course be interesting to find examples that flesh out
Corollary \ref{main eq teorrrr}. Recent work on heat kernels
on metric spaces (such as fractals) is promising. These examples illustrate that there is a heirarchy of conditions on the asymptotics  of zeta functions and heat kernels.
In the study of diffusion processes on fractals \cite{K} one is given
the generator of the heat semigroup $\Delta$ as a positive self adjoint densely defined operator.
Then we assume that  $\zeta_\Delta^{s}=\tau(\Delta^{-s/2})<\infty$ for
all $s>p$ where $p$ is called the spectral dimension. (For simplicity we are going to assume
that $\Delta$ has bounded inverse for if not there is a simple remedy \cite{CPS}.)

The weakest condition we can impose is that there are constants $C_0, C_0'$ with:
\begin{equation} \label{zetab} C_0'\leq (s-p)\zeta_\Delta(s)\leq C_0\end{equation}
for all $s>p$.
It follows from Theorem 4.5 of \cite{CRSS} that the operator $\Delta^{-p/2}\in\mathcal{M}_{1,\infty}$  however, it also follows by
an example in \cite{CRSS} (which does not come from any concrete diffusion process but is an artificial counterexample) that this bound is insufficient to obtain heat kernel
bounds and that the best we can do is the bound (\ref{hk marc}) where we need to insert the Cesaro mean.

On the other hand a heat kernel bound of the form
\begin{equation}\label{bound}
C^{-1}t^{-p/2}\leq\tau(e^{-t\Delta})\leq C t^{-p/2},\quad 0<t<1
\end{equation}
(which is known to hold for some diffusion processes on metric spaces
and in particular for certain fractals, see for example \cite{K})  is stronger than the zeta function bound (\ref{zetab}) as can
be seen by the following elementary argument.

Recall that if $B\in\mathcal{M}$ is a positive operator then it follows from
the  spectral theorem that
$$B^{s/2}=\frac1{\Gamma(s/2)}\int_0^{\infty}t^{s/2-1}e^{-tB^{-1}}dt.$$
Next, suppose that (\ref{bound}) holds.   Setting $B=\Delta^{-1}$, it follows from \eqref{bound} that for all $s>p$ we have
\begin{equation}\label{split}\int_0^{\infty}t^{s/2-1}\tau(e^{-t\Delta})dt\leq
 C\int_0^1t^{(s-p)/2-1}dt+\int_1^{\infty}t^{s/2-1}\tau(e^{-\Delta})e^{-(t-1)\|\Delta^{-1}\|^{-1}}dt\end{equation}
$$\leq\frac{2C}{s-p}+e^{\|\Delta^{-1}\|^{-1}}\tau(e^{-\Delta})\frac{\Gamma(s/2)}{\|\Delta^{-1}\|^{s/2}}.$$
It follows from Fatou lemma that $B^{s/2}$ is trace class for all $s>p$ and so
$$(s-p)\zeta_\Delta(s)\leq 2C+o(1),\quad s\downarrow p.$$
Similarly, we have  $$\int_0^{\infty}t^{s/2-1}\tau(e^{-t\Delta})dt\geq
 C^{-1}\int_0^1t^{(s-p)/2-1}dt=\frac{2C^{-1}}{s-p}$$
and therefore
$$2C^{-1}\leq(s-p)\zeta_\Delta(s)
\quad s\downarrow p.$$

We have assumed that the function $s\to\tau(\Delta^{-s/2})$ is analytic in $s$ for $\Re(s)>p$ and that it may have a singularity at $s=p.$ However, we saw in Lemma \ref{mero} that the nature of this singularity is not obvious in general.

It is well known (and, in the context of the questions discussed here, explained in \cite{CRSS}) how an asymptotic expansion for small $t$ of the form
$ \tau(e^{-t\Delta})\sim Ct^{-p/2}+ O(t^{-\alpha/2})$, where $\alpha<p$,
implies that the $\zeta$-function has a meromorphic continuation to a half plane $\Re(s)>p-\epsilon$, for some $\epsilon>0$ with
the only singularity in this half plane being a simple pole at $s=p$. 
However such an assumption is not in line with what has been found for certain
fractals. 

There is a discussion of
the pole structure of the zeta function for certain fractal diffusion processes in
\cite{DGV}, \cite{ST}, and literature cited therein.
There we find fractals where
$\zeta_\Delta$ is meromorphic with simple poles on the line $\{p+iv\vert v\in \mathbb R\}$.
To discuss this situation we can employ here a well known argument similar to that of \cite{CRSS}, in particular the ideas introduced
 in the proof of Theorem 5.2  in that paper (where we used  the notation $T=\Delta^{-1}$). 

We have
$\zeta_\Delta(s)=\int_0^{\infty}t^{s/2-1}\tau(e^{-t\Delta})dt$ and may split this integral into two parts
as in (\ref{split}). Then
only
$\int_0^{1}t^{s/2-1}\tau(e^{-t\Delta})dt$
contributes to the singularity at $s=p$ (as we exploited  in \cite{CRSS}).
Now suppose that we have a simple asymptotic expansion of the form $ tr(e^{-t\Delta}\sim Ct^{-p/2}+ O(t^{-\alpha/2})$
with $\alpha<p$ for $0<t<1$. Then
$$\int_0^{1}t^{s/2-1}\tau(e^{-t\Delta})dt=\frac{2C}{s-p} +G(s)$$
where $G(s)=\int_0^1 (t^{s/2-1}\tau(e^{-t\Delta})-Ct^{(s-p)/2})dt$. Here
the integrand is by assumption continuous and $O(t^{(s-\alpha)/2})$ and hence $G$ is analytic for $\Re(s)>\alpha$
 and in particular on the line
$\{p+iv\vert v\in \mathbb R\}$ which is inconsistent with the assumption of there being poles
on this line. Thus the asymptotic behaviour of the trace of the heat kernel
must be more complicated for such fractals.

There is a positive result that we obtain from
Corollary \ref{parallel}.
Setting $t=\lambda^{-2}$ and $A=\Delta^{-1/2}$ in the formula in this Corollary gives
$$\lim_{t\to 0}M(t\to\tau(t^{p/2}e^{-t \Delta}))=\frac12\Gamma(\frac{p}2)\lim_{z\to0}z\tau(\Delta^{-(p+z)/2}).$$
Connecting with our previous notation we set $s=p+z$
and see that the presence of a simple pole at $s=p$ for the zeta function means that
$\lim_{t\to 0}M(t\to\tau(t^{p/2}e^{-t \Delta}))$
exists.
This simple pole behavior at $s=p$ is  conjectured in \cite{ST}
to be a generic feature of a
certain class of fractals. Our Corollary \ref{parallel} suggests
that to infer from this, information about the trace of the heat kernel
for small $t$, it is
more promising to investigate the asymptotic behavior of
\begin{equation}\label{Ces} M(t\to t^{p/2}\tau(e^{-t \Delta})).
\end{equation}

To illustrate this we use \cite{ABS}. There it is shown that
for the Sierpinski Gasket one has for small $t$,
\begin{equation}\label{Serp}\tau(e^{-t\Delta})= t^{-\beta}\gamma(t) +o(t^{-\beta}) 
\end{equation}
where $\beta= \log 3/\log 5$ and

\begin{equation}\label{gamma}\gamma(t)= \sum_{-\infty}^{\infty} c_n\Gamma(1+\beta +\frac{2\pi i n}{\log 5}) e^{-2\pi i n \log t/log 5}.\end{equation}
Numerical evidence supports the conjecture that
$\gamma(t) =a+b \sin\frac{2\pi}{\log 5}(\log t-c)$ for some real $a,b,c$.
If this conjecture is true we can insert equation (\ref{Serp})
into equation (\ref{Ces}). Then we make the change of variable $t=\lambda^{-2}$
and consider
 the resulting Cesaro mean (\ref{Cesaro}) as a function of the asymptotic variable $\nu$.
We see that it is bounded by $a+\frac{C}{\log \nu}$
for some constant $C$ as $\nu \to \infty$ so that the $t$ independent constant $a$
in $\gamma(t)$ gives the required zeta function residue.
%



\end{document}